\documentclass[12pt, reqno]{amsart}
\usepackage{amsmath}
\usepackage{amsthm}
\usepackage{a4wide}
\usepackage{amsfonts,amsmath,amsthm,amssymb}
\usepackage[english]{babel}
\usepackage[T1]{fontenc}
\usepackage{amsfonts}
\usepackage{graphicx}
\usepackage{fancyhdr}
\usepackage{verbatim}

\usepackage{color}

\makeatother

\newtheorem{teo}{Theorem}[section]
\newtheorem{cor}[teo]{Corollary}
\newtheorem{lem}[teo]{Lemma}
\newtheorem{prop}[teo]{Proposition}
\newtheorem{re}[teo]{Remark}
\newtheorem{defi}[teo]{Definition}

\newcommand{\res}{\mathop{\hbox{\vrule height 7pt width .5pt depth 0pt
\vrule height .5pt width 6pt depth 0pt}}\nolimits}
\newcommand{\R}{\mathbb{R}}
\newcommand{\eps}{\varepsilon}

\title{SBV regularity for Hamilton-Jacobi equations in $\R^n$}

\author{Stefano Bianchini}
\address{Stefano Bianchini,
SISSA, via Beirut 2-4 34014 Trieste,
Italy}
\author{Camillo De Lellis}
\author{Roger Robyr}
\address{Camillo De Lellis, Roger Robyr,
Institut f\"ur Mathematik, Universit\"at Z\"urich,
Winterthurerstrasse 190, CH-8057
Z\"urich, Switzerland}

\email{bianchin@sissa.it, delellis@math.unizh.ch,
roger.robyr@math.unizh.ch}

\begin{document}
\begin{abstract}
In this paper we study the regularity of viscosity solutions
to the following Hamilton-Jacobi equations
$$
\partial_t u + H(D_{x} u)=0 \qquad \textrm{in } \Omega\subset
\R\times \R^{n}\, .
$$
In particular, under the assumption that the Hamiltonian
$H\in C^2(\R^n)$ is uniformly convex, we prove that $D_{x}u$
and $\partial_t u$ belong to the class $SBV_{loc}(\Omega)$.
\end{abstract}

 \maketitle

\section{Introduction}

In this paper, we consider viscosity solutions
$u$ to Hamilton-Jacobi equations
\begin{equation}\label{HJ initial}
\partial_{t}u+H(D_{x} u)=0 \qquad \textrm{in } \Omega\subset
[0,T]\times \R^{n} .
\end{equation}
As it is well known, solutions of the Cauchy problem for
\eqref{HJ initial} develop singularities of the gradient
in finite time, even if the initial data $u (0, \cdot)$
is extremely regular.
The theory of viscosity solutions, introduced by Crandall and Lions
30 years ago, provides several
powerful existence and uniqueness results which allow
to go beyond the formation of singularities. Moreover,
viscosity solutions are the limit of several smooth
approximations of \eqref{HJ initial}.
For a review of the concept of viscosity solution and the
related theory for equations of type \eqref{HJ initial}
we refer to \cite{bress1,cansin,lions}. 

In this paper we are concerned about the regularity
of such solutions, under the following key assumption:
\begin{equation}\label{convex ineq}
H\in C^2(\R^n) \qquad \mbox{and} \qquad 
c_{H}^{-1}Id_{n}\leq D^2H\leq c_{H}Id_{n} \quad\mbox{for some
$c_H>0$.}
\end{equation}
There is a vast literature about this issue. As it is
well-known, under the assumption \eqref{convex ineq}, any viscosity
solution $u$ of \eqref{HJ initial} is locally semiconcave in $x$. More precisely,
for every $K\subset\subset \Omega$ 
there is a constant $C$ (depending on $K, \Omega$ and $c_H$) such that
the function $x\mapsto u(t,x)- C |x|^2$ is concave on $K$. This easily implies that
$u$ is locally Lipschitz and that $\nabla u$ has locally bounded variation,
i.e. that the distributional Hessian $D^2_x u$
is a symmetric matrix of Radon measures. It is then not difficult
to see that the same conclusion holds for $\partial_t D_x u$ and
$\partial_{tt} u$. Note that this result is
independent of the boundary values of $u$ and can be regarded as an
interior regularization effect of the equation.

The rough intuitive picture
that one has in mind is therefore that of functions which are Lipschitz
and whose gradient is piecewise smooth, undergoing jump discontinuities
along a set of codimension $1$ (in space and time). A refined regularity
theory, which confirms this picture and goes beyond, analyzing the behavior
of the functions where singularities are formed, is available under 
further assumptions on the boundary values of $u$ (we refer
to the book \cite{cansin} for an account on this research topic). However, if the boundary values are
just Lipschitz, these results do not apply and the corresponding 
viscosity solutions might be indeed quite rough, if we understand 
their regularity only in a pointwise sense. 

In this paper we prove that the BV regularization
effect is in fact more subtle and there is a measure-theoretic analog of ``piecewise
$C^1$ with jumps of the gradients''. As a consequence of our analysis, 
we know for instance that the singular parts of the Radon
measures $\partial_{x_ix_j} u$, $\partial_{x_it} u$
and $\partial_{tt} u$ are concentrated on a rectifiable set of 
codimension $1$. This set is indeed the measure theoretic jump set $J_{D_x u}$
of $D_x u$ (see below for the precise definition). This
excludes, for instance, that the second derivative of $u$ 
can have a complicated
fractal behaviour. Using the language introduced in \cite{dgamb} we say that
$D_x u$ and $\partial_t u$ are (locally) special functions of bounded variation, i.e.
they belong to the space $SBV_{loc}$ (we refer to the monograph \cite{afp}
for more details). A typical example of a 
$1$-dimensional function which belongs to BV but not to SBV is the
classical Cantor staircase (cp. with Example 1.67 of \cite{afp}).

\begin{teo}\label{main theo}
Let $u$ be a viscosity
solution of \eqref{HJ initial}, assume \eqref{convex ineq}
and set $\Omega_{t}:=\{x\in \R^n:(t,x)\in
\Omega\}$.
Then, the set of times
\begin{equation}\label{e:exceptional}
S:=\{t:D_{x}u(t,.) \notin SBV_{loc}(\Omega_{t})\}
\end{equation}
is at most countable. In particular $D_x u, \partial_t u \in SBV_{loc} (\Omega)$.
\end{teo}

\begin{cor}\label{corollary1}
Under assumption \eqref{convex ineq}, 
the gradient of any viscosity solution $u$ of
\begin{equation}
H(D_{x}u)=0\qquad \textrm{in } \Omega\subset \R^{n},
\end{equation}
belongs to $SBV_{loc}(\Omega)$.
\end{cor}

Theorem \ref{main theo} was proved first by Luigi Ambrosio
and the second author in the special case $n=1$ (see \cite{adl}
and also \cite{robyr} for the extension to Hamiltonians
$H$ depending on $(t,x)$ and $u$). Some of the ideas
of our proof originate indeed in the work \cite{adl}. However,
in order to handle the higher dimensional case, some new
ideas are needed. In particular, a key role is played
by the geometrical theory of monotone functions developed
by Alberti and Ambrosio in \cite{aa}.

\section{Preliminaries: the theory of monotone functions}
\begin{defi}
Let $\Omega\subset \R^{n}$ be an open set. We say that a continuous
function $u: \Omega\rightarrow \R$ is \emph{semiconcave}
if, for any convex $K \subset\subset \Omega$, there exists
$C_{K}> 0$ such that
\begin{equation}\label{e:semiconc-cond}
u(x+h)+u(x-h)-2u(x)\leq C_{K}|h|^{2},
\end{equation}
for all $x,h \in \R^n$ with $x,x-h,x+h \in K$.
The smallest nonnegative costant $C_K$ such that \eqref{e:semiconc-cond}
holds on $K$ will be called {\em semiconcavity constant of $u$ on $K$}.
\end{defi}

Next, we introduce the concept of superdifferential.
\begin{defi}
Let $u: \Omega \to \R$ be a measurable function.
The set $\partial u(x)$,
called the
$\emph{superdifferential}$
of $u$ at point $x\in \Omega$, is defined as
\begin{align}
\partial u(x)&:=\Big{\{}p\in \R^n: \limsup_{y\rightarrow
x}\frac{u(y)-u(x)-p\cdot(y-x)}{|y-x|}\leq 0\Big{\}}.
\end{align}
\end{defi}

Using the above definition we can describe some properties of
semiconcave functions (see Proposition 1.1.3 of \cite{cansin}):
\begin{prop}\label{concave proposition}
Let $\Omega\subset \R^n$ be open and $K\subset \Omega$ a 
compact convex set. Let $u:\Omega\rightarrow
\R$ be a semiconcave function with semiconcavity constant
 $C_K \geq 0$. Then, the function
\begin{equation}
\tilde{u}:x\mapsto u(x)-\frac{C_K}{2}|x|^2 \qquad \textrm{ is concave in }K.
\end{equation}
In particular, for any given $x,y\in K$, $p\in \partial \tilde u(x)$ and
$q\in \partial \tilde{u} (y)$ we have that
\begin{equation}
\langle q-p,y-x\rangle \leq 0.
\end{equation}
\end{prop}

From now on, when $u$ is a semi--concave function,
we will denote the set-valued map $x\to \partial \tilde{u} (x) + C_K x$
as $\partial u$. An important observation is that, being $\tilde{u}$
concave, the map $x\to \partial \tilde{u} (x)$ is a maximal monotone
function.

\subsection{Monotone functions in $\R^{n}$}\label{chap mononotone}

Following the work of Alberti and Ambrosio \cite{aa} we introduce
here some results about the theory of monotone functions in
$\R^{n}$. Let $B:\R^n\rightarrow \R^n$ be a set-valued map (or
multifunction), i.e. a map which maps every point $x\in \R^n$ into
some set $B(x)\subset \R^n$. For all $x\in \R^n$ we define:
\begin{itemize}
  \item the {\em domain} of $B$, $Dm (B):=\{x:B(x)\neq\emptyset\}$,
  \item the {\em image} of $B$, $Im (B):=\{y: \exists x, y\in B(x)\}$,
  \item the {\em graph} of $B$,
$\Gamma B:=\{(x,y)\in \R^n \times \R^n: y\in
  B(x)\}$,
  \item then {\em inverse} of $B$, $[B^{-1}](x):=\{y:x\in B(y)\}$.
\end{itemize}

\begin{defi}\label{def monotone fct} Let $B:\R^n \rightarrow \R^n$
be a multifunction, then
\begin{enumerate}
    \item $B$ is a \emph{monotone} function if
    \begin{equation}\label{e:minoreuguale}
    \langle y_{1}-y_{2},x_{1}-x_{2}\rangle\leq 0 \qquad \forall
    x_{i} \in \R^{n}, y_{i} \in B(x_{i}), i=1,2.
    \end{equation}
    \item A monotone function $B$ is called \emph{maximal} when it
    is maximal with respect to the inclusion in the class of
    monotone functions, i.e. if the following implication holds:
    \begin{equation}
    A(x)\supset B(x) \textrm{ for all }x, A \textrm{ monotone
    }\Rightarrow A=B.
    \end{equation}
\end{enumerate}
\end{defi}

Observe that in this work we assume $\leq$ in \eqref{e:minoreuguale}
instead of the most common $\geq$. However, one can pass
from one convention to the other by simply considering $-B$ instead of
$B$. The observation of the previous subsection is then
summarized in the following Theorem.

\begin{teo}\label{t:concave}
The supergradient $\partial u$ of a concave function is a maximal
monotone function.
\end{teo}

An important tool of the theory of maximal monotone functions,
which will play a key role in this paper, is the
Hille-Yosida approximation
(see Chapters 6 and 7 of \cite{aa}):
\begin{defi}\label{d:Hille}
For every $\varepsilon>0$
we set
$\Psi_{\varepsilon}(x,y):=(x-\varepsilon y,y)$ for all $(x,y)\in
\R^n \times \R^n$,
and for every maximal monotone function $B$ we
define $B_{\varepsilon}$ as the multifunction whose graph is
$\Psi_{\varepsilon}(\Gamma B)$,
that is, $\Gamma B_{\varepsilon}=\{(x-\varepsilon y,y):
(x,y)\in \Gamma B\}$. Hence
\begin{equation}
B_{\varepsilon}:=(\varepsilon Id-B^{-1})^{-1}.
\end{equation}
\end{defi}

In the next Theorems we collect some properties of maximal
monotone functions $B$ and their approximations
$B_{\varepsilon}$ defined above.

\begin{teo}\label{t:cont}
Let $B$ be a maximal monotone function.
Then, the set
$S(B):= \{x: B (x) \mbox{ is not single valued}\}$
is a $\mathcal{H}^{n-1}$ rectifiable set.
Let $\tilde{B}: Dm (B)\to \R^n$ be such that $\tilde{B} (x) \in
\tilde{B} (x)$ for every $x$.
Then $\tilde{B}$ is a measurable function and $B(x)=\{\tilde{B} (x)\}$
for a.e. $x$.
If $Dm (B)$ is open, then $D\tilde{B}$ is a measure, i.e. $\tilde{B}$ is a
function of locally bounded variation. 

If $K_i$ is a sequence of compact sets contained
in the interior of $Dm (B)$ with $K_i\downarrow
K$, then $B (K_i)\to B(K)$ in the Hausdorff sense.
Therefore, the map $\tilde{B}$ is continuous at every $x\not\in S(B)$.

Finally, if $Dm (B)$ is open and $B=\partial u$ for
some concave function $u: Dm (B)\to \R$, then $\tilde{B}(x)
= Du (x)$ for a.e. $x$ (recall that $u$ is locally
Lipschitz, and hence the distributional
derivative of $u$ coincides a.e. with the classical differential).
\end{teo}

\begin{proof} First of all, note that, by Theorem 2.2 of \cite{aa},
$S(B)$ is the union of rectifiable sets of Hausdorff dimension 
$n-k$, $k\geq 1$. This guarantees the existence of the classical
measurable function $\tilde{B}$. The BV regularity when
$Dm (B)$ is open is shown in Proposition 5.1 of \cite{aa}.

\medskip

Next, let $K$ be a compact set contained in the interior of $Dm (B)$.
By Corollary 1.3(3) of \cite{aa}, $B (K)$ is bounded. Thus,
since $\Gamma B \cap K\times \R^n$s is closed by maximal
monotonicity, it turns out that it is also compact. The continuity
claimed in the second paragraph of the Theorem is then a simple
consequence of this observation.


\medskip

The final paragraph of the Theorem is proved in Theorem 7.11 of \cite{aa}.
\end{proof}

In this paper, since we will always consider monotone functions
that are the supergradients of some concave functions, we will
use $\partial u$ for the supergradient and $Du$ for the distributional
gradient. A corollary of Theorem \ref{t:cont} is that

\begin{cor}\label{c:contconc}
If $u: \Omega \to \R$ is semiconcave, then $\partial u (x)
= \{Du (x)\}$ for a.e. $x$, and at any point where
$\partial u$ is single--valued, $Du$ is continuous. Moreover
$D^2 u$ is a symmetric matrix of Radon measures.
\end{cor}

Next we state the following important convergence theorem.
For the notion of current and the corresponding convergence
properties we refer to the work of Alberti and Ambrosio.
However, we remark that very little of the theory
of currents is needed in this paper: what we actually need is a simple corollary of
the convergence in (ii), which is stated and proved in Subsection \ref{ss:approx}.
In (iii) we follow the usual convention
of denoting by $|\mu|$
the total variation of a (real-, resp. matrix-,
vector- valued) measure $\mu$. The theorem stated below is in fact
contained in Theorem 6.2 of \cite{aa}.

\begin{teo}\label{convergence teo}
Let $\Omega$ be an open and convex subset of $\R^n$ and
let $B$ be a maximal monotone function such that
$\Omega\subset Dm(B)$. Let
$B_{\varepsilon}$ be the approximations given in Definition
\ref{d:Hille}. Then, the following properties hold.
\begin{enumerate}
    \item[(i)] $B_{\varepsilon}$ is a $1/\varepsilon$-Lipschitz maximal
    monotone function on $\R^n$ for every $\varepsilon>0$. Moreover,
if $B= Du$, then $B_\eps = Du_\eps$ for the concave function 
\begin{equation}\label{e:esplicita}
 u_\eps (x) \;:=\; \inf_{y\in \R^n} \left\{ u (y) + \frac{1}{2\eps} |x-y|^2\right\}
\end{equation}
    \item[(ii)] $\Gamma B$ and $\Gamma B_\eps$
have a natural structure as integer rectifiable currents,
and $\Gamma B_{\varepsilon} \res
\Omega \times \R^n$ converges to $\Gamma B \res \Omega\times \R^n$ in
the sense of currents
as $\varepsilon\downarrow 0$.
    \item[(iii)] $DB_{\varepsilon}\rightharpoonup^* D\tilde{B}$ and
   $|DB_{\varepsilon}|\rightharpoonup^*
    |D\tilde{B}|$ in the sense of measures on $\Omega$.
\end{enumerate}
\end{teo}

\subsection{BV and SBV functions} We conclude the
section by introducing the basic notations related to
the space $SBV$ (for a complete survey on this topic we address the
reader to \cite{afp}).
If $B\in BV(A, R^k)$, then it is possible to split the measure $DB$ into
three mutually singular parts:
$$DB=D_{a}B+D_{j}B+D_{c}B.$$
$D_{a}B$ denotes the absolutely continuous part (with respect to the
Lebesgue measure). $D_{j}B$ denotes the jump part of $D B$.
When $A$ is a $1$-dimensional domain, $D_j B$ consists of
a countable sum of weighted Dirac masses, and hence it is also
called the atomic part of $DB$. In higher dimensional domains, $D_j B$
is concentrated on a rectifiable set of codimension $1$, which corresponds
to the measure-theoretic jump set $J_B$ of $B$.
$D_{c}B$ is called the
Cantor part of the gradient and it is the ``diffused part''
of the singular measure $D_{s}B:=D_{j}B+D_{c}B$. Indeed 
\begin{equation}\label{e:cantor_prop}
D_c B (E) = 0
\qquad\mbox{for any Borel set $E$ with $\mathcal{H}^{n-1} (E)<\infty$.}
\end{equation} 
For all these
statements we refer to Section 3.9 of \cite{afp}.

\begin{defi}
Let $B\in BV(\Omega)$, then $B$ is a special function of bounded
variation, and we write $B\in SBV(\Omega)$, if $D_{c}B=0$, i.e. if the
measure $DB$ has no Cantor part. The more general
space $SBV_{loc} (\Omega)$ is
defined in the obvious way.
\end{defi}

In what follows, when $u$ is a (semi)-concave function, we will
denote by $D^2 u$ the distributinal hessian of $u$. Since $D u$
is, in this case, a $BV$ map, the discussion above applies.
In this case we will use the notation $D^2_a u$, $D^2_j u$ and
$D^2_c u$. An important property of $D^2_c u$ is the
following regularity property.

\begin{prop}\label{p:cont2}
Let $u$ be a (semi)-concave function. If $D$ denotes the set
of points where $\partial u$ is not single--valued, then
$|D^2_c u| (D) =0$.
\end{prop}
\begin{proof} By Theorem \ref{t:cont}, the set $D$ is $\mathcal{H}^{n-1}$-rectifiable.
This means in particular, that it is $\mathcal{H}^{n-1}-\sigma$ finite. By the property
\eqref{e:cantor_prop} we conclude $D^2_c u (E)=0$ for every Borel subset $E$ of $D$.
Therefore $|D^2_c u| (D) =0$.
\end{proof}

\section{Hamilton-Jacobi equations}

In this section we collect some definitions and well-known results
about Hamilton-Jacobi equations. For a complete survey on this
topic we redirect the reader to the vast literature. 
For an introduction to the topic we suggest the following sources
\cite{bress1},\cite{cansin},\cite{evans}. In this paper we will
consider the following Hamilton-Jacobi equations
\begin{eqnarray}
&&\partial_{t} u+H(D_{x}u)\;=\; 0, \qquad \hbox{in $\Omega \subset
[0,T]\times \R^n $\, ,}
\label{e:HJt}\\
&& H (D_x u) \;=\; 0, \qquad\,\,\,\,\qquad \hbox{in 
$\Omega\subset \R^n$\, ,
}\label{e:HJs}
\end{eqnarray}
under the assumption that
\begin{itemize}
    \item[{\bf A1:}] The Hamiltonian $H\in C^2(\R^n)$ satisfies:
    $$p \mapsto H(p) \textrm{ is convex and }\lim_{|p|\rightarrow \infty} \frac{H(p)}{|p|}=+\infty.$$
\end{itemize}
Note that this assumption is obviously implied by \eqref{convex ineq}.

We will often consider $\Omega= [0,T]\times \R^n$ in \eqref{e:HJt}
and couple it with the initial condition
\begin{equation}\label{e:initial}
u (0,x)\;=\; u_0 (x)
\end{equation}
under the assumption that
\begin{itemize}
\item[{\bf A2:}]
    The initial data $u_{0}:\R^n\rightarrow\R$  is Lipschitz continuous and bounded.
\end{itemize}
\begin{defi}[Viscosity solution]
A bounded, uniformly continuous function $u$ is called a
\emph{viscosity solution} of \eqref{e:HJt} (resp. \eqref{e:HJs})
provided that
\begin{enumerate}
    \item $u$ is a \emph{viscosity subsolution} of \eqref{e:HJt} (resp. \eqref{e:HJs}): for
    each $v\in C^{\infty}(\Omega)$ such that $u-v$
    has a maximum at $(t_{0},x_{0})$ (resp. $x_0$),
    \begin{equation}
    v_{t}(t_{0},x_{0})+H(D_{x}v(t_{0},x_{0}))\leq 0
\qquad \mbox{(resp. $H (D v (x_0))\leq 0$);}
    \end{equation}
    \item $u$ is a \emph{viscosity supersolution} of \eqref{e:HJt} (resp. \eqref{e:HJs}): for
    each $v\in C^{\infty}(\Omega)$ such that $u-v$
    has a minimum at $(t_{0},x_{0})$ (resp. $x_0$),
    \begin{equation}
    v_{t}(t_{0},x_{0})+H(D_{x}v(t_{0},x_{0}))\geq 0
\qquad\mbox{(resp. $H (Dv (x_0))\geq 0$).}
    \end{equation}
\end{enumerate}
In addition, we say that
$u$ solves the Cauchy problem \eqref{e:HJt}-\eqref{e:initial} on
$\Omega = [0,T]\times \R^n$ if \eqref{e:initial} holds in the
classical sense.
\end{defi}
\begin{teo}[The Hopf-Lax formula as viscosity solution]\label{viscosity-theorem}
The unique viscosity solution of the initial-value problem
\eqref{e:HJt}-\eqref{e:initial} is given by the Hopf-Lax formula
\begin{equation}\label{Hopf-Lax}
u(t,x)=\min_{y\in \R^n}
\Big{\{}u_{0}(y)+tL\Big{(}\frac{x-y}{t}\Big{)} \Big{\}} \qquad (t>0,
x\in \R^{n}),
\end{equation}
where $L$ is the Legendre transform of $H$:
\begin{equation}
L(q):=\sup_{p\in \R^{n}} \{p\cdot q- H(p)\} \qquad (q \in \R^{n}).
\end{equation}
\end{teo}
In the next Proposition we collect some properties of the
viscosity solution defined by the Hopf-Lax formula:
\begin{prop}\label{properties visc sol}
Let $u(t,x)$ be the viscosity solution of \eqref{e:HJt}-\eqref{e:initial} and defined by
\eqref{Hopf-Lax}, then
\begin{enumerate}
    \item[(i)] \textbf{A functional identity:} For each $x\in\R^n$ and
    $0\leq s <t \leq T$, we have
    \begin{equation}\label{functional identity}
    u(t,x)=\min_{y\in \R^n}
    \Big{\{}u(s,y)+(t-s)L\Big{(}\frac{x-y}{t-s}\Big{)} \Big{\}}.
    \end{equation}
\item[(ii)] \textbf{Semiconcavity of the solution:} For any
fixed $\tau>0$ there exists a constant $C(\tau)$
such that the function defined by
    \begin{equation}
    u_t:\R^n \rightarrow \R^n 
\textrm{ with } u_t(x):=u(t,x),
    \end{equation}
    is semiconcave with constant less than $C$ for any $t\geq \tau$.
\item[(iii)]
\textbf{Characteristics:} The minimum point $y$ in \eqref{Hopf-Lax}
is unique if and only if $\partial u_t (x)$ is single valued. Moreover,
in this case we have $y=x-t DH(D_x u (t,x))$.
\item[(iv)] \textbf{The linear programming principle:} Let $t>s>0$, $x\in \R^n$ and
assume that $y$ is a minimum for \eqref{Hopf-Lax}. Let $z= 
\frac{s}{t} x + (1-\frac{s}{t}) y$. 
Then $y$ is the {\em unique} minimum for $u_0 (w) + s L
((z-w)/s)$.
\end{enumerate}
\end{prop}

\begin{re}
For a detailed proof of Theorem \ref{viscosity-theorem} and
Proposition \ref{properties visc sol} we address the reader to
Chapter 6 of \cite{cansin} and Chapters 3, 10 of \cite{evans} .
\end{re}

Next, we state a useful locality property of the solutions of \eqref{e:HJt}.

\begin{prop}\label{p:locality}
Let $u$ be a viscosity solution of \eqref{e:HJt} in $\Omega$. Then $u$ is locally
Lipschitz. Moreover, for any $(t_0, x_0)\in \Omega$, there exists a neighborhood
$U$ of $(t_0, x_0)$, a positive number $\delta$ 
and a Lipschitz function $v_0$ on $\R^n$ such that
\begin{itemize}
\item[(Loc)] $u$ coincides on $U$ with the viscosity solution of
\begin{equation}\label{e:(Loc)}
\left\{
\begin{array}{l}
\partial_t v + H (D_x v) \;=\; 0 \qquad \mbox{in $[t_0-\delta, \infty[\times \R^n$}\\ \\
v (t_0-\delta, x) \;=\; v_0(x)\, .
\end{array}\right.
\end{equation}
\end{itemize}
\end{prop}

This property of viscosity solutions of Hamilton-Jacobi equations is 
obviously related to the finite speed of propagation (which holds
when the solution is Lipschitz) and it is well-known. One could
prove it, for instance, suitably modifying the proof of Theorem 7
at page 132 of \cite{evans}. On the other hand
we have not been able to find a complete reference for Proposition
\ref{p:locality}. Therefore, for the reader's convenience, we provide
a reduction to some
other properties clearly stated in the literature. 

\begin{proof} 
The local Lipschitz regularity of $u$ follows from its local semiconcavity,
for which we refer to \cite{cansin}. As for the locality property
(Loc), we let $\delta>0$ and $R$ be such that $C:=
[t_0-\delta, t_0+\delta]\times 
\overline{B}_R (x_0)\subset \Omega$. It is then known that
the following dynamic programming principle holds for every $(t,x)\in C$
(see for instance Remark 3.1 of \cite{canson}
or \cite{evsoug}):
\begin{eqnarray}
u (t,x) &=& \inf \bigg\{ \int_\tau^t L (\dot{\xi} (s))\, ds 
+ u (\tau, \xi (\tau)) \, \Big|\, \tau\leq t, \xi\in W^{1,\infty} ([\tau, t]),
\label{e:dyn_prog}\\
&&\qquad\qquad\qquad\qquad\qquad\qquad\qquad \xi (t)=x
\mbox{ and } (\tau, \xi (\tau))\in \partial C\bigg\}\, . \nonumber
\end{eqnarray}
The Lipschitz regularity of $u$ and the convexity of $L$
ensure that a minimizer exists. Moreover any minimizer is a straight line.
Next, assume that $x\in B_\delta (x_0)$. If $\delta$ is much smaller than $R$, 
the Lipschitz regularity
of $u$ ensures that any minimizer $\xi$ has the endpoint $(\tau, \xi (\tau))$
lying in $\{t_0-\delta\}\times B_R (x_0)$. Thus,
for every $(t, x)\in [t_0-\delta, t_0+\delta]\times B_\delta (x_0)$
we get the formula
\begin{equation}\label{e:HL_loc}
u(t,x) \;=\; \min_{y\in \overline{B}_R (x_0)}
\left(u(t_0-\delta, y) + (t-t_0+\delta) L \left(\frac{x-y}{t-t_0+\delta}\right)\right)\, .
\end{equation}
Next, extend the map $\overline{B}_R (0)\ni x\mapsto u (t_0-\delta, x)$ to
a bounded Lipschitz map $v_0: \R^n \to \R$, keeping the same Lipschitz constant.
Then the solution of \eqref{e:(Loc)} is given by the Hopf-Lax formula
\begin{equation}\label{e:HL_glob}
v (t,x)\;=\; \min_{y\in \R^n} \left( v_0 (y) + (t-t_0+\delta)
L \left(\frac{x-y}{t-t_0+\delta}\right)\right)\, .
\end{equation}
If $(t, x)\in [t_0-\delta, t_0+\delta]\times B_\delta (0)$,
then any minimum point $y$ in \eqref{e:HL_glob} belongs to $\overline{B}_R (0)$,
provided $\delta$ is sufficiently small (compared to $R$ and the Lipschitz
constant of $v$, which in turn is bounded independently of $\delta$).
Finally, since $v_0 (y) = u (t_0-\delta, y)$ for every $y\in \overline{B}_R (0)$,
\eqref{e:HL_loc} and \eqref{e:HL_glob} imply that $u$ and $v$
coincide on $[t_0-\delta, t_0+\delta]\times B_\delta (0)$ provided $\delta$
is sufficiently small.
\end{proof}

\section{Proof of the main Theorem}

\subsection{Preliminary remarks}
Let $u$ be a viscosity solution of \eqref{e:HJt}. By Proposition
\ref{p:locality} and the time invariance of the equation, we can,
without loss of generality, assume that $u$ is a solution on
$[0,T]\times \R^n$ of the Cauchy-Problem
\eqref{e:HJt}-\eqref{e:initial} under the assumptions A1, A2.
Clearly, it suffices to show that, for every $j>0$, the set of times
$S\cap ]1/j, +\infty[$ is countable. Therefore, by Proposition
\ref{viscosity-theorem} and the time--invariance of 
the Hamilton--Jacobi
equations, we can restrict ourselves to the following case:
\begin{equation}\label{2der bounded}
\mbox{$\exists C$ s.t. $u_{\tau}$ is semiconcave
with constant less than $C$ and $|Du_\tau|\leq C$ $\forall \tau\in [0,T]$}.
\end{equation}
Arguing in the same way, we can further assume that
\begin{equation}\label{e:eps}
\mbox{$T$ is smaller than some constant $\eps (C)>0$,}
\end{equation}
where the choice of the constant $\eps (C)$ will be specified later.

Next we consider a ball $B_R (0)\subset \R^n$ and a bounded
convex set $\Omega\subset [0,T]\times \R^n$ with the properties
that:
\begin{itemize}
\item $B_R (0)\times \{s\}\subset \Omega$ for every $s\in [0,T]$;
\item For any $(t,x)\in \Omega$ and for any $y$ reaching the
minimum in the formulation \eqref{Hopf-Lax}, $(0,y)\in \Omega$
(and therefore the entire segment joining $(t,x)$ to $(0,y)$
is contained in $\Omega$).
\end{itemize}
Indeed, recalling that $\|Du\|_\infty <\infty$, it suffices 
to choose $\Omega:= \{(x,t)\in \R^n\times [0,T]: |x|\leq R + C' (T-t)\}$
where the costant $C'$ is sufficiently large,
depending only on $\|Du\|_\infty$ and $H$.  
Our goal is now to show
the countability of the set $S$ in \eqref{e:exceptional}.

\subsection{A function depending on time}
For any $s<t\in [0,T]$, we define the set--valued map
\begin{equation}\label{e:X}
X_{t,s} (x) \;:=\; x- (t-s) DH (\partial u_t (x))\, .
\end{equation}
Moreover, we will denote by $\chi_{t,s}$ the restriction
of $X_{t,s}$ to the points where $X_{t,s}$ is single--valued.
According to Theorem \ref{t:cont} and Proposition
\ref{properties visc sol}(iii),
the domain of $\chi_{t,s}$ consists of
those points where $Du_t (\cdot)$ is continuous, which are those
where the minimum point $y$ in \eqref{functional identity}
is unique. Moreover, in this case we have $\chi_{t,s} (x) = \{y\}$.

Clearly, $\chi_{t,s}$ is defined a.e. on $\Omega_t$. With 
a slight abuse of notation we set
\begin{equation}
F(t)\;:=\; |\chi_{t,0} (\Omega_t)|\, ,
\end{equation}
meaning that, if we denote by $U_t$ the set of points $x\in
\Omega_{t}$ such that \eqref{Hopf-Lax} has a unique minimum point,
we have $F(t) = |X_{t,0} (U_t)|$.

The proof is then split in the following three lemmas:

\begin{lem}\label{nonincreasing}
The functional $F$ is nonincreasing,
\begin{equation}
F(\sigma)\geq F(\tau)  \qquad\mbox{for any $\sigma, \tau
\in [0,T]$ with $\sigma<\tau$}.
\end{equation}
\end{lem}

\begin{lem}\label{cantorlem}
If $\eps$ in \eqref{e:eps} is small enough, then the following holds.
For any $t\in]0,T[$ and $\delta\in ]0, T-t]$
there exists a Borel set $E \subset \Omega_{t}$ such that
\begin{itemize}
\item[(i)] $|E|=0$, and
$|D^2_c u_t| (\Omega_t\setminus E)=0$;
\item[(ii)] $X_{t,0}$ is single valued on $E$
(i.e. $X_{t,0} (x) = \{\chi_{t,0} (x)\}$ for every $x\in E$);
\item[(iii)] and
\begin{equation}\label{cantor disappear from X}
\chi_{t,0} (E) \cap \chi_{t+\delta, 0}
(\Omega_{t+\delta})=\emptyset.
\end{equation}
\end{itemize}
\end{lem}

\begin{lem}\label{ineqlem2}
If $\eps$ in \eqref{e:eps} is small enough, then the following holds.
For any $t\in ]0,\varepsilon]$ and any Borel set 
$E\subset \Omega_{t}$, we have
\begin{equation}\label{ineqq}
|X_{t,0} (E)| \geq c_{0} |E|-c_1 t \int_{E}
d(\Delta u_{t})\, ,
\end{equation}
where $c_{0}$ and $c_1$ are positive constants
and $\Delta u_{t}$
is the Laplacian of $u_{t}$.
\end{lem}

\subsection{Proof of Theorem \ref{main theo}}\label{ss:main}
The three key lemmas stated above will be proved in the next two sections.
We now show how to complete the proof of the Theorem.
First of all, note that $F$ is a bounded function. Since
$F$ is, by Lemma \ref{nonincreasing}, a monotone function,
its points of discontinuity are, at most, countable.
We claim that,
if $t\in ]0,T[$ is such that $u_t\not\in SBV_{loc} (\Omega_t)$,
then $F$ has a discontinuity at $t$.

Indeed, in this case we have
\begin{equation}\label{e:noSBV}
|D^2_{c}u_{t}| (\Omega_{t})> 0.
\end{equation}
Consider any $\delta>0$ and let $B=E$ be the set of
Lemma \ref{cantorlem}. Clearly, by Lemma \ref{cantorlem}(i) and (ii),
\eqref{cantor disappear from X} and \eqref{ineqq},
\begin{equation}\label{e:quasi}
F(t+\delta)\;\leq\; F(t) + c_1 t \int_E d\, \Delta_s u_t \;\leq\; F(t) +
c_1 t \int_{\Omega_t} d\, \Delta_c u_t\, ,
\end{equation}
where the last inequality follows from
$\Delta_s u_t = \Delta_c u_t + \Delta_j u_t$
and $\Delta_j u_t\leq 0$ (because of the semiconcavity of $u$).

Next, consider the Radon--Nykodim decomposition $D^2_c u_t = M
|D^2_c u_t|$, where $M$ is a matrix--valued Borel function with
$|M|=1$. Since we are dealing with second derivatives, $M$ is
symmetric, and since $u_t$ is semiconcave, $M\leq 0$. Let
$\lambda_1, \ldots, \lambda_n$ be the eigenvalues of $- M$. Then
$1=|M|^2 = \lambda_1^2 + \ldots + \lambda_n^2$ and $- Tr M =
\lambda_1+\ldots + \lambda_n$. Since $\lambda_i\geq 0$, we easily
get $- Tr M \geq 1$. Therefore,
\begin{equation}\label{e:CS}
-\Delta_{c}u_{t}\;=\; - Tr M |D^2_c u_t| \;\geq\; |D^2_c u_t|\, .
\end{equation}
Hence
$$
F(t+\delta)\;\stackrel{\eqref{e:quasi}+\eqref{e:CS}}{\leq}\; F(t) -
c_1 t|D^2_c u_t| (\Omega_t)\, .
$$
Letting $\delta\downarrow 0$ we conclude
$$
\limsup_{\delta \downarrow 0} F(t+\delta)\;<\; F(t)\, .
$$
Therefore $t$ is a point of discontinuity of $F$, which is the desired
claim.

\subsection{Easy corollaries}
The conclusion that $D_x u \in SBV (\Omega)$ follows from the slicing
theory of $BV$ functions (see Theorem 3.108 of \cite{afp}).
In order to prove the same property for
$\partial_t u$ we apply the Volpert chain rule
to $\partial_t u = - H (D_x u)$. According to Theorem 3.96 of
\cite{afp}, we conclude that $[\partial_{x_j t}]_c u = 
- \sum_i \partial_i H (D_x u) [\partial_{x_jx_i}]_c u = 0$ (because $[D^2_x]_c u = 0$)
and $[\partial_{tt}]_c u = - \sum_i \partial_i H (D_x u) [\partial_{x_i t}]_c u = 0$
(because we just concluded $[D^2_{x t}]_c u =0$).

As for Corollary \ref{corollary1}, let $u$ be a viscosity solution of
\eqref{e:HJs} and set $\widetilde{u}(t,x):=u(x)$. Then $\tilde{u}$
is a viscosity solution of
$$
\partial_{t}\widetilde{u}+H(D_{x}\widetilde{u})=0\,
$$
in $\R\times \Omega$.
By our main Theorem \ref{main theo} the set of
times for which $D_x \widetilde{u}(t,.)\notin SBV_{loc}(\Omega)$ is
at most countable. Since $D_x \widetilde{u} (t, \cdot) = D u$,
for every $t$, we conclude that $Du \in
SBV_{loc}(\Omega)$.

\begin{re}
The special case of this Corollary for $\Omega\subset \R^2$ was
already proved in \cite{adl} (see Corollary 1.4 therein).
We note that the
proof proposed in \cite{adl} was more complicated than the one above.
This is due to the power of Theorem \ref{main theo}.
In \cite{adl} the authors proved the $1$--dimensional case
of Theorem \ref{main theo}. The proof above reduces the $2$--dimensional
case of Corollary \ref{corollary1} to the $2+1$ case of Theorem
\ref{main theo}. In \cite{adl} the $2$-dimensional case of Corollary
\ref{corollary1} was reduced to the $1+1$ case of Theorem
\ref{main theo}: this reduction requires a subtler argument.
\end{re}

\section{Estimates}

In this section we prove two important estimates. The first
is the one in Lemma \ref{ineqlem2}. The second is an estimate
which will be useful in proving Lemma \ref{cantorlem} and
will be stated here.

\begin{lem}\label{ineqlem1}
If $\eps (C)$ in \eqref{e:eps} is sufficiently small,
then the following holds.
For any $t\in ]0,T]$, any $\delta \in [0,t]$ and any Borel set
$E\subset \Omega_{t}$ we have
\begin{equation}\label{ineqlem3}
\Big{|}X_{t, \delta} (E)\Big{|}\geq \frac{(t-\delta)^n}{t^n}
 \Big{|}X_{t,0} (E)\Big{|}\, .
\end{equation}
\end{lem}

\subsection{Injectivity}
In the proof of both lemmas, the following remark
plays a fundamental role.

\begin{prop}\label{p:inj} For any $C>0$
there exists $\eps (C)>0$ with the following property. If $v$ is a
semiconcave function with constant less than $C$, then the map
$x\mapsto x-t DH (\partial v)$ is injective for every $t\in [0, \eps
(C)]$.
\end{prop}
Here the injectivity of a set--valued map $B$
is understood in the following natural way
$$
x\neq y\qquad \Longrightarrow \qquad B(x)\cap B(y)=\emptyset\, .
$$
\begin{proof}
We assume by contradiction that there exist $x_{1},x_{2}\in
\Omega_{t}$ with $x_1\neq x_2$ and such that:
$$
[x_1 - t DH (\partial v (x_1))]\cap [x_2 - t DH (\partial v
(x_2))]\neq\emptyset.
$$
This means that there is a point $y$ such that
\begin{equation}
\left\{%
\begin{array}{ll}
    \frac{x_{1}-y}{t} \in DH(\partial v (x_{1})), \\
    \frac{x_{2}-y}{t} \in DH(\partial v (x_{2})); \\
\end{array}%
\right.
\Rightarrow
\left\{%
\begin{array}{ll}
    DH^{-1}(\frac{x_{1}-y}{t})\in\partial v (x_{1}), \\
    DH^{-1}(\frac{x_{2}-y}{t})\in\partial v (x_{2}). \\
\end{array}%
\right.
\end{equation}
By the semiconcavity of $v$ we get:
\begin{equation}\label{ineq1}
M(x_{1},x_{2}):=\Big{\langle}
DH^{-1}\Big{(}\frac{x_{1}-y}{t}\Big{)}-DH^{-1}\Big{(}\frac{x_{2}-y}{t}
\Big{)},x_{1}-x_{2}\Big{\rangle}
\leq C |x_1-x_2|^2.
\end{equation}
On the other hand, $D (DH^{-1}) (x) = (D^2 H)^{-1} (DH^{-1} (x))$ (note that in this
formula, $DH^{-1}$ denotes the inverse of the map $x\mapsto DH (x)$, whereas
$D^2 H^{-1} (y)$ denotes the matrix $A$ which is the inverse of the matrix $B:=
D^2 H (y)$).
Therefore $D(DH^{-1}) (x)$ is a symmetric matrix,
with $D(DH^{-1}) (x)\geq c_H^{-1} Id_n$.
It follows that
\begin{align}\label{ineq2}
M(x_{1},x_{2})&=t\Big{\langle}
DH^{-1}\Big{(}\frac{x_{1}-y}{t}\Big{)}-DH^{-1}\Big{(}\frac{x_{2}-y}{t}\Big{)},
\frac{x_{1}-y}{t}-\frac{x_{2}-y}{t}\Big{\rangle}\geq
\nonumber \\
&\geq
\frac{t}{2c_{H}}\Big{|}\frac{x_{1}-y}{t}-\frac{x_{2}-y}{t}
\Big{|}^2\geq
\frac{1}{2 t c_{H}}|x_{1}-x_{2}|^2\geq \frac{1}{2\varepsilon
c_{H}}|x_{1}-x_{2}|^2.
\end{align}
But if $\varepsilon>0$ is small enough, or more precisely if it is
chosen to satisfy $2\varepsilon c_{H}< \frac{1}{C}$ the two
inequalities (\ref{ineq1}) and (\ref{ineq2}) are in contradiction.
\end{proof}

\subsection{Approximation}\label{ss:approx}
We next consider $u$ as in the formulations of
the two lemmas, and $t\in [0,T]$. Then the function
$\tilde{v} (x):= u (x) - C|x|^2/2$ is concave.
Consider the approximations $B_\eta$ (with $\eta>0$) of $\partial \tilde{v}$
given in Definition \ref{d:Hille}. By Theorem
\ref{convergence teo}(i), $B_\eta = D \tilde{v}_\eta$ for
some concave function $\tilde{v}_\eta$ with Lipschitz gradient. Consider therefore
the function $v_\eta (x) = \tilde{v}_\eta (x) + C|x|^2/2$.
The semiconcavity constant of $v_\eta$ is not larger than $C$.

Therefore we can apply Proposition \ref{p:inj} and
choose $\eps (C)$ sufficiently small in such a way that
the maps
\begin{equation}\label{e:girate}
x\;\mapsto\; A(x) = x - t DH (\partial u_t)
\qquad
x\;\mapsto\; A_\eta (x)= x - t DH (Dv_\eta)
\end{equation}
are both injective. Consider next the following measures:
\begin{equation}\label{e:measures}
\mu_\eta (E) \;:=\; |(Id - t DH (Dv_\eta)) (E)| \qquad \mu (E)\;:=\;
|(Id - t DH (\partial u_t)) (E)|\, .
\end{equation}
These measures are well-defined because of the injectivity
property proved in Proposition \ref{p:inj}.

Now, according to Theorem \ref{convergence teo},
the graphs $\Gamma Dv_\eta$ and
$\Gamma \partial u_t$ are both rectifiable currents
and the first are converging, as $\eta\downarrow 0$, to
the latter. We denote them, respectively,
by $T_\eta$ and $T$. Similarly, we can associate
the rectifiable currents $S$ and $S_\eta$ to the graphs
$\Gamma A$ and $\Gamma A_\eta$ of
the maps in \eqref{e:girate}. Note that these graphs can
be obtained by composing $\Gamma \partial u_t$
and $\Gamma Dv_\eta$ with the following
global diffeomorphism of $\R^n$:
$$
(x,y)\;\mapsto\; \Phi (x,y) = x-tDH(y)\, .
$$
In the language of currents we then have
$S_\eta = \Phi_\sharp T_\eta$ and $S=\Phi_\sharp
T$. Therefore, $S_\eta \to S$ in the sense of currents.

We want to show that
\begin{equation}\label{e:convergence}
\mu_\eta \;\rightharpoonup^*\; \mu\, .
\end{equation}
First of all, note that $S$ and $S_\eta$ are rectifiable currents of
multiplicity $1$ supported on the rectifiable sets $\Gamma A = \Phi
(\Gamma \partial u_t)$ and $\Gamma A_\eta = \Phi(\Gamma B_\eta) =
\Phi (\Gamma Dv_\eta)$. Since $B_\eta$ is a Lipschitz map, the
approximate tangent plane $\pi$  to $S_\eta$ in (a.e.)
point $(x, A_\eta (x))$ is spanned by the vectors $e_i + DA_\eta
(x)\cdot e_i$ and hence oriented by the $n$-vector 
$$
\stackrel{\to}{v} \;:=\;
\frac{(e_1 + DA_\eta (x)\cdot e_1)\wedge \ldots \wedge (e_n+DA_\eta (x)\cdot e_n)}
{|(e_1 + DA_\eta (x)\cdot e_1)\wedge \ldots \wedge (e_n+DA_\eta (x)\cdot e_n)|}\, .
$$
Now, by the calculation of Proposition \ref{p:inj},
it follows that $\det DA_\eta\geq 0$. Hence
\begin{equation}\label{e:positivity}
\langle dy_1\wedge \ldots \wedge dy_n,
\stackrel{\to}{v}\rangle \;\geq\; 0\, .
\end{equation}
By the convergence $S_\eta\to S$,
\eqref{e:positivity} holds for the tangent planes
to $S$ as well.

Next, consider a
$\varphi\in C^\infty_c (\Omega_t)$. Since both $\Gamma A$ and
$\Gamma A_\eta$ are bounded sets, consider a
ball $B_R (0)$ such that ${\rm supp}\, (\Gamma A), {\rm supp}\,
(\Gamma A_\eta)\subset \R^n\times B_R (0)$ and
let $\chi\in C^\infty_c (\R^n)$ be a cut-off function
with $\chi|_{B_R (0)} = 1$.
Then, by standard calculations on currents,
the injectivity property of Proposition \ref{p:inj}
and \eqref{e:positivity} imply that
\begin{eqnarray}
\int \varphi d\mu
&=& \langle S, \varphi (x) \chi (y) dy_1\wedge\ldots\wedge dy_n\rangle,\\
\int \varphi d\mu_\eta
&=& \langle S_\eta, \varphi (x) \chi (y) dy_1\wedge\ldots\wedge dy_n
\rangle\, .
\end{eqnarray}
Therefore, since $S_\eta \to S$, we conclude that
$$
\lim_{\eta\downarrow 0} \int \varphi d\mu_\eta
 \;=\; \int \varphi d\mu\, .
$$
This shows \eqref{e:convergence}.

\subsection{Proof of Lemma \ref{ineqlem1}}\label{ss:prima}
First of all we choose $\eps$ so small that the
conclusions of Proposition \ref{p:inj} and
those of Subsection \ref{ss:approx} hold.

We consider therefore, the approximations $v_\eta$
of Subsection \ref{ss:approx}, we define the measures
$\mu$ and $\mu_\eta$ as in \eqref{e:measures}
and the measures $\hat{\mu}$ and $\hat{\mu}_\eta$ as
\begin{equation}\label{e:measures2}
\hat{\mu} (E) \;:=\;
|(Id - (t-\delta) DH (\partial u_t)) (E)|
\qquad \hat{\mu}_\eta (E) \;:=\;
|(Id - (t-\delta) DH (Dv_\eta)) (E)|\, .
\end{equation}
By the same arguments as in Subsection \ref{ss:approx},
we necessarily have $\hat{\mu}_\eta\rightharpoonup^* \hat{\mu}$.

The conclusion of the Lemma can now be formulated as
\begin{equation}\label{e:restated}
\hat{\mu} \;\geq\; \frac{(t-\delta)^n}{t^n}
\mu\, .
\end{equation}
By the convergence of the measures $\mu_\eta$ and
$\hat{\mu}_\eta$ to $\mu$ and $\hat{\mu}$, it suffices
to show
\begin{equation}\label{e:restated10}
\hat{\mu}_\eta \;\geq\; \frac{(t-\delta)^n}{t^n}
\mu_\eta\, .
\end{equation}
On the other hand, since the maps $x\mapsto x-t DH (Dv_\eta)$ and
$x\mapsto x- (t-\delta) DH (Dv_\eta)$ are both injective and
Lipschitz, we can use the area formula to write:
\begin{eqnarray}
\hat{\mu}_\eta (E) &=&
\int_{E}\det\Big{(}Id_n-(t-\delta) D^{2}H(D v_\eta (x))
D^2 v_\eta (x)\Big{)}\, dx,\\
\mu_\eta (E)
&=&  \int_{E}\det\Big{(}Id_n-t D^{2}H(D v_\eta (x))
D^2 v_\eta (x)\Big{)}\, dx
\end{eqnarray}
Therefore, if we set
\begin{eqnarray*}
M_1 (x) &:=& Id_n - (t-\delta) D^2H (Dv_\eta (x))
D^2 v_\eta (x)\\
M_2 (x) &:=& Id_n - t D^2H (Dv_\eta (x))D^2 v_\eta (x)\, \, ,
\end{eqnarray*}
the inequality \eqref{e:restated} is equivalent to
\begin{equation}\label{e:ineqdet}
\det M_1 (x)\;\geq\; \frac{(t-\delta)^n}{t^n} \det M_2 (x) \qquad
\mbox{for a.e. $x$.}
\end{equation}
Note next that
\begin{eqnarray*}
\det M_1 (x) &=&
\det (D^2 H (Dv_\eta (x)))
\det\Big{(}[D^2H(D v_\eta (x))]^{-1}-(t-\delta)
D^2 v_\eta (x)\Big{)}\nonumber\\
\det M_2 (x) &=&
\det (D^2 H (Dv_\eta (x)))
\det\Big{(}[D^2H(Dv_\eta (x))]^{-1}-tD^2 v_\eta (x)\Big{)}
\end{eqnarray*}
Set $A (x):= [D^2H(D v_\eta (x))]^{-1}$ and $B (x)= D^2 v_\eta (x)$.
Then it suffices to prove that:
\begin{equation}
\det (A (x)-(t-\delta) B (x))\;\geq\;
\frac{(t-\delta)^n}{t^n}
\det (A(x)- tB (x))\, .
\end{equation}
Note that
$$
A - (t-\delta) B \;=\; \frac{\delta}{t} A + \frac{t-\delta}{t}
(A-tB)\, .
$$
By choosing $\eps$ sufficiently small (but only depending
on $c_H$ and $C$),
we can assume that $A-tB$ is a positive semidefinite matrix.
Since $A$ is a positive definite matrix, we conclude
\begin{equation}\label{e:>=}
A - (t-\delta) B \;\geq\; \frac{t-\delta}{t}
(A-tB)\, .
\end{equation}
A standard argument in linear algebra shows that
\begin{equation}\label{e:det>=}
\det (A-(t-\delta) B)\;\geq\; \frac{(t-\delta)^n}{t^n}
\det (A-tB)\,
\end{equation}
which concludes the proof.
We include, for the reader convenience, a proof of \eqref{e:>=}
$\Longrightarrow$\eqref{e:det>=}. It suffices to show that,
if $E$ and $D$ are positive semidefinite matrices with $E\geq D$,
then $\det E\geq \det D$. Without loss of generality,
we can assume that $E$ is in diagonal form, i.e.
$E= {\rm diag}\, (\lambda_1, \ldots, \lambda_n)$, and that
$E>D$. Then each
$\lambda_i$ is positive. Define $G:={\rm diag}\, (\sqrt{\lambda_1}
,\ldots, \sqrt{\lambda_n})$. Then
$$
{\rm Id}_n \;\geq\; G^{-1} D G^{-1} = \tilde{D}\, .
$$
Our claim would follow if we can prove $1\geq \det \tilde{D}$, that
is, if we can prove the original claim for $E$ and $D$ in the
special case where $E$ is the identity matrix. But in this case we
can diagonalize $E$ and $D$ at the same time. Therefore $D= {\rm
diag}\, (\mu_1, \ldots, \mu_n)$. But, since $E\geq D\geq 0$, we have
$0\leq \mu_i\leq 1$ for each $\mu_i$. Therefore
$$
\det E\;=\; 1 \;\geq\; \Pi_i \mu_i \;=\; \det D\, .
$$

\subsection{Proof of Lemma \ref{ineqlem2}}
As in the proof above we will show the Lemma by approximation
with the functions $v_\eta$. Once again we introduce
the measures $\mu_\eta$ and $\mu$ of \eqref{e:measures}.
Then, the conclusion of the Lemma can be formulated as
\begin{equation}\label{e:restated2}
\mu\;\geq\; c_0\mathcal{L}^n - t c_1 \Delta u_t\, .
\end{equation}
Since $\Delta v_\eta \rightharpoonup^* \Delta u_t$
by Theorem \ref{convergence teo}(iii), it suffices to
show
\begin{equation}\label{e:restated20}
\mu_\eta\;\geq\; c_0\mathcal{L}^n - t c_1 \Delta v_\eta\, .
\end{equation}
Once again we can use the area formula to
compute
\begin{equation}\label{estimate 0}
\mu_\eta (E) \;=\; \int_{E} \det (D^2 H (Dv_\eta (x)))
\det\Big{(}[D^2H(D v_\eta (x))]^{-1}-tD^2 v_\eta(x)\Big{)} dx
\end{equation}
Since $D^2 H\geq c_H^{-1} Id_n$ and $[D^2 H]^{-1}\geq c_H^{-1}
Id_n$, we can estimate
\begin{equation}\label{estimate 1}
\det (D^2 H (Dv_\eta (x))) \det\Big{(}[D^2H(D v_\eta (x))]^{-1}-tD^2
v_\eta(x)\Big{)} \;\geq\; c_H^{-n} \det \left(\frac{1}{c_H} Id_n - t
D^2 v_\eta (x)\right)
\end{equation}
arguing as in Subsection \ref{ss:prima}. If we choose $\varepsilon$
so small that $0<\varepsilon <\frac{1}{2c_{H}C}$, then $M (x)
:= \frac{1}{2c_{H}}Id_n-tD^{2} v_\eta (x)$ is positive semidefinite.
Therefore
\begin{equation}\label{estimate 2}
\det (D^2 H (Dv_\eta (x))) \det\Big{(}[D^2H(D v_\eta (x))]^{-1}-tD^2
v_\eta(x)\Big{)} \;\geq\; c_H^{-n} \det \left(\frac{1}{2c_H} Id_n +
M (x)\right)\,.
\end{equation}
Diagonalizing $M (x) = {\rm diag}
(\lambda_1, \ldots, \lambda_n)$,
we can estimate
\begin{eqnarray}
\det \left(\frac{1}{2c_{H}}Id_{n}+ M (x)\right) &=&
\left(\frac{1}{2c_H}\right)^n  \prod_{i=1}^{n} (1+ 2c_H \lambda_i)
\;\geq\; \left(\frac{1}{2c_H}\right)^n (1 + 2c_H {\rm Tr}\, M (x))
\nonumber\\
&=& c_2 - c_3 t \Delta v_\eta (x)\, .\label{e:estimate2}
\end{eqnarray}
Finally, by \eqref{estimate 0},
\eqref{estimate 1}, \eqref{estimate 2} and
\eqref{e:estimate2}, we get
$$
\mu_\eta (E) \;\geq\; \int_E (c_0 - c_1t \Delta v_\eta (x))\,
dx\,.
$$
This concludes the proof.

\section{Proofs of Lemma \ref{nonincreasing} and
Lemma \ref{cantorlem}}

\subsection{Proof of Lemma \ref{nonincreasing}}
The claim follows from the following consideration:
\begin{equation}\label{e:inclusion}
\chi_{t,0} (\Omega_t) \subset \chi_{s,0} (\Omega_s)
\qquad \mbox{for every $0\leq s\leq t\leq T$.}
\end{equation}
Indeed, consider $y\in \chi_{t,0} (\Omega_t)$. Then there exists
$x\in \Omega_t$ such that $y$ is the unique minimum of
\eqref{Hopf-Lax}. Consider $z:= \frac{s}{t} x + \frac{t-s}{t} y$.
Then $z\in \Omega_s$. Moreover, by Proposition \ref{properties visc
sol}(iv), $y$ is the unique minimizer of $u_0 (w) + s L ((z-w)/ s)$.
Therefore $y=\chi_{s,0} (z)\in \chi_{s,0} (\Omega_s)$.

\subsection{Proof of Lemma \ref{cantorlem}}
First of all, by Proposition \ref{p:cont2},
we can select a Borel set $E$ of measure $0$ such that
\begin{itemize}
\item $\partial u_t (x)$ is single-valued for every $x\in E$;
\item $|E|=0$;
\item $|D^2_c u_t| (\Omega_t\setminus E)=0$.
\end{itemize}
If we assume that our statement were false, then there would exist
a compact set $K\subset E$ such that 
\begin{equation}\label{e:contra4}
|D^2_c u_t| (K)\;>\; 0\, .
\end{equation}
and $X_{t,0} (K) = \chi_{t,0} (K)
\subset \chi_{t+\delta, 0} (\Omega_{t+\delta})$.
Therefore it turns out that $X_{t,0} (K)
= \chi_{t+\delta,0} (\tilde{K}) = X_{t+\delta, 0} (\tilde{K})$ for some
Borel set $\tilde{K}$. 

Now, consider $x\in \tilde{K}$
and let $y := \chi_{t+\delta, 0} (x)\in
X_{t+\delta, 0} (\tilde{K})$ and $z := \chi_{t+\delta, t} (x)$.
By
Proposition \ref{properties visc sol}(iv), $y$
is the unique minimizer of $u_0 (y) + t L ((z-y)/t)$,
i.e. $\chi_{t, 0} (z) = y$.

Since $y\in \chi_{t, 0} (K)$, there exists
$z'$ such that $\chi_{t,0} (z')$. On the other hand,
by Proposition \ref{p:inj}, provided $\eps$ has been
chosen sufficiently small,
$\chi_{t,0}$ is an injective map. Hence we necessarily have
$z'=z$. This shows that
\begin{equation}\label{e:inclusion2}
X_{t+\delta, t} (\tilde{K}) \subset K\, .
\end{equation}
By Lemma \ref{ineqlem1},
\begin{equation}\label{e:contra1}
|K| \;\geq\; |X_{t+\delta, t} (\tilde{K})|
\;\geq\; \frac{\delta^n}{(t+\delta)^n}
|X_{t+\delta, 0} (\tilde{K})|
\;=\; \frac{\delta^n}{(t+\delta)^n} |X_{t,0} (K)|\, .
\end{equation}
Hence, by Lemma \ref{ineqlem2}
\begin{equation}\label{e:contra2}
|K|\;\geq\; c_0 |K| - c_1 t \frac{\delta^n}{(t+\delta)^n}
\int_K d\, \Delta u_t\, .
\end{equation}
On the other hand, recall that $K\subset E$ and $|E|=0$. Thus,
$\int_K d\,\Delta_s u_t = \int_K d\, \Delta u_t \geq 0$.
On the other hand $\Delta_s u_t \leq 0$ (by the semiconcavity
of $u$). Thus we conclude that $\Delta_s u_t$, and hence
also $\Delta_c u_t$, vanishes indentically on $K$.
However, arguing as in Subsection \ref{ss:main},
we can show $-\Delta_c u_t \geq |D^2_c u_t|$, and hence,
recalling \eqref{e:contra4}, $-\Delta_c u_t (K)>0$.
This is a contradiction
and hence concludes the proof.

\end{document}